\documentclass{au}
\usepackage{amssymb,latexsym}
\usepackage{amsmath}
\usepackage[dvipdfm]{graphicx}

\newcommand{\bgq}{\boldsymbol{\theta}}

\newcommand{\gf}{\varphi}
\newcommand{\setm}[2]{\{#1:#2\}}
\newcommand{\setbm}[2]{\bigl\{#1:#2\bigr\}}
\newcommand{\set}[1]{\{#1\}}
\newcommand{\tbf}{\textbf}
\newcommand \alg[1]{{\mathcal #1}}
\newcommand \bset[1]{{{\textup{Block}}(#1)}}
\newcommand \Tol[1] {\textup{Tol}(#1)}
\newcommand \var[1]{\pmb{\mathcal #1}}
\newcommand\lat{\pmb{{L}at}}
\newcommand\tlat{\pmb{T\kern-1.5 pt Lat}}
\newcommand \Set{\pmb{{S}\kern-0.5pt et}} 
\newcommand \setnn[1] {\var S^{(#1)}} 
\newcommand \setni[2] {\var S^{(#1)}_{#2}} 
\newcommand \anyvar[1]{\var V_{#1}}
\newcommand \nopvar{\var{V}}
\newcommand \rotlat[1] {\pmb{RLat}_{#1}} 
\newcommand \rottamp[2] {^{#2}\kern-1.5pt(#1\kern-1.5pt)}
\newcommand \setnntamp[2] {(#1\kern-1.5pt)^{\kern-1pt{#2}}}
\newcommand \combvar[2] {\pmb{C}_{\kern-1pt#1#2}}

\newcommand \indterm {t}

\newcommand \rotop[1]{g_{#1}}
\newcommand \setop[1] {f_{#1}}
\newcommand \sophop {h}
\newcommand \sterm {s}
\newcommand \tjoin{t_{\vee}}
\newcommand \tmeet{t_{\wedge}}

\newcommand\proponetfac{\textup{(P1)}}
\newcommand\proptwostfac{\textup{(P2)}}
\newcommand\propthreetcp{\textup{(P3)}}
\newcommand\propfourprop{\textup{(P4)}}
\newtheorem{theorem}{Theorem}
\newtheorem{proposition}[theorem]{Proposition}
\newtheorem{lemma}[theorem]{Lemma}

\theoremstyle{definition}

\newtheorem{problem}[theorem]{Problem}
\newtheorem{example}[theorem]{Example}
\newtheorem*{ackno}{Acknowledgment}
\newenvironment{enumeratei}{\begin{enumerate}[\upshape (i)]} {\end{enumerate}}
\begin{document}
\title{Independent joins of tolerance factorable varieties}

\author[I.\ Chajda]{Ivan Chajda}
\email{ivan.chajda@upol.cz}
\address{Palack\'y University Olomouc\\Department of Algebra and Geometry\\17. listopadu 12,
771 46 Olomouc, Czech Republic}

\author[G.\ Cz\'edli]{G\'abor Cz\'edli}
\email{czedli@math.u-szeged.hu}
\urladdr{http://www.math.u-szeged.hu/~czedli/}
\address{University of Szeged\\ Bolyai Institute\\Szeged,
Aradi v\'ertan\'uk tere 1\\ Hungary 6720}

\author[R.\ Hala\v s]{Radom\'\i r Hala\v s}
\email{radomir.halas@upol.cz} 
\address{Palack\'y University Olomouc\\Department of Algebra and Geometry\\17. listopadu 12,
771 46 Olomouc, Czech Republic}

\thanks{This research was supported the project Algebraic Methods in Quantum Logic, No.:
CZ.1.07/2.3.00/20.0051, 
by the NFSR of Hungary (OTKA), grant numbers   K77432 and K83219, and by T\'AMOP-4.2.1/B-09/1/KONV-2010-0005}

\dedicatory{Dedicated to B\'ela Cs\'ak\'any on his eightieth birthday}

%


\date{May 11, 2012}

\begin{abstract} 
Let $\lat$ denote the variety of lattices. In 1982, the second author proved that $\lat$ is \emph{{strongly} tolerance factorable}, that is, the members of $\lat$ have quotients in $\lat$ modulo tolerances, although $\lat$ has proper tolerances. We did not know any other nontrivial example of a strongly tolerance factorable variety. Now we prove that this property is preserved by forming independent joins (also called products) of varieties. This enables us to present infinitely many {strongly} tolerance factorable varieties with proper tolerances. Extending a recent result of G.\ Cz\'edli and G.\ Gr\"atzer, we show that  if $\var V$ is a strongly tolerance factorable variety, then the tolerances of $\var V$ are exactly the homomorphic images of  congruences of algebras in $\var V$. Our observation that  (strong) tolerance factorability is not necessarily preserved when passing from a variety to an equivalent one leads to an open problem.
\end{abstract}

\maketitle

\subsection*{Basic concepts}
Given an algebra $\alg A = (A, F)$, a binary reflexive, symmetric, and compatible
relation $T \subseteq A\times A=A^2$ is called a \emph{tolerance} on $\alg A$.  
The set of tolerances of $\alg A$ is denoted by $\Tol{\alg A}$.
A tolerance  which is not a congruence is called \emph{proper}. By a \emph{block} of a tolerance $T$ we mean  a maximal subset $B$ of $A$ such that $B^2\subseteq T$. 
 Let $\bset T$ denote the set of all blocks of $T$. It follows from Zorn's lemma that, for  $X\subseteq A$, we have that
\begin{equation}\label{HeRZrnLypp}
\text{$X^2\subseteq T$ if{f} $X\subseteq
 U$ for some $U\in \bset T$.}
\end{equation} 
Applying this observation to $X=\{a,b\}$, we obtain that $\bset T$ determines $T$. 
Furthermore, we also conclude that, 
for each $n$, each $n$-ary $f \in F$, and all $B_1,\dots,B_n\in\bset T$, there exists a $B\in\bset T$ such that 
\begin{equation}\label{mfaXCDsF}
\setm{f(b_1,\ldots,b_n)} {b_1\in B_1,\ldots, b_n\in B_n}\subseteq B\text.
\end{equation}
We say that $\alg A$ is \emph{$T$-factorable} if, for each $n$, each $n$-ary $f \in F$ and all $B_1,\dots,B_n\in\bset T$, 
the block $B$ in \eqref{mfaXCDsF} is uniquely determined. 
In this case, we  define $f(B_1,\ldots, B_n) := B$, and we call the algebra
$(\bset T,F)$ the \emph{quotient algebra}
$\alg A/T$ of $\alg A$  \emph{modulo the tolerance} $T$. 
{If $\alg A$ is $T$-factorable for all $T\in \Tol A$, then we say that $\alg A$ is \emph{tolerance factorable}.}  In what follows, we focus on  the following properties of varieties; $\var V$ denotes a variety of algebras. The \emph{tolerances} of $\var V$ are understood as the tolerances of algebras of $\var V$.
\begin{list}{}{\setlength{\leftmargin}{1.6cm}\setlength{\rightmargin}{1cm}}
\item[\proponetfac] 
$\var V$ is \emph{ tolerance factorable} if all of its members are tolerance factorable.
\item[\proptwostfac] 
$\var V$ is \emph{ strongly  tolerance factorable} if it is tolerance factorable and, 
for all  $\alg A\in \var V$ and all  $T\in\Tol{\alg A}$, $\alg A/T\in \var V$. 
\item[\propthreetcp] 
\emph{The tolerances of $\,\var V$ are the images of {its} congruences} if for each $\alg A\in\var V$ and every $T\in\Tol{\alg A}$, there exist an algebra $\alg B\in \var V$, a congruence $\bgq$ of $\alg B$ and a surjective homomorphism $\gf\colon \alg B\to \alg A$ such that 
$T=\setm{(\gf(a),\gf(b))} {(a,b)\in\bgq}$.
\item[\propfourprop] 
$\var V$ has proper tolerances if at least one of its members has a proper tolerance. %
\end{list}

\emph{Term equivalence}, in short, \emph{equivalence}, of varieties was introduced by W.\,D.~Neumann~\cite{neumann}. (He called it rational equivalence.) Instead of recalling the technical definition, we mention that the variety of Boolean algebras is  equivalent to that of Boolean rings. The variety of sets (with no operations) is denoted by $\Set$. 
Although the present paper is self-contained, for more information on tolerances   the reader is referred to the monograph I.~Chajda~\cite{chajdabook} .

\subsection*{Motivation and the target}
{Besides $\lat$ {and $\Set$}, no other  {strongly tolerance factorable variety with proper tolerances}  has been  known since 1982.}
Our initial goal was to find {some other ones}. {We prove} that independent joins, see later, preserve {each of the properties \proponetfac{}--\propfourprop{}}. This enables us to construct infinitely many, pairwise non-equivalent,  strongly tolerance factorable varieties with proper tolerances. Also, we show that if a variety is strongly tolerance factorable, then its tolerances are the images of its congruences, but the converse implication fails. Finally, we show that (strong) tolerance factorability is not always preserved when passing from a variety to an equivalent one, and we raise 
an open problem based on this fact.

\subsection*{Independent joins}
Let $n\in\mathbb N=\set{1,2,\ldots}$, and let $\anyvar 1,\ldots,\anyvar n$ be varieties of the same type. These varieties are called \emph{independent} if there exists an $n$-ary term $\indterm$ in their common type such that, for $i=1,\ldots,n$, 
$\anyvar i$ satisfies the identity $\indterm(x_1,\ldots,x_n)=x_i$. In this case, the join $\nopvar$ of 
the varieties $\anyvar 1,\ldots,\anyvar n$ is called an \emph{independent join} (in the lattice of all varieties of a given type). This concept was introduced by G.~Gr\"atzer, H.~Lakser, and J.~P\l onka~\cite{gglakplonka}. Independent joins of varieties are also called (direct) \emph{products}.

\begin{proposition}[W.~Taylor~\cite{taylor}, G.~Gr\"atzer, H.~Lakser, and J.~P\l onka \cite{gglakplonka}]\label{propIndep}
Assume that a variety $\nopvar$ is the independent join of its subvarieties $\anyvar 1, \cdots, \anyvar n$. 
\begin{enumeratei}
\item \label{xhTalma1} Every algebra $\alg A\in \nopvar$ is $($isomorphic to$)$ a product $\alg A_1\times\cdots\times \alg A_n$ with $\alg A_1\in\anyvar 1$, \dots, $\alg A_n\in\anyvar n$. These $A_i$ are uniquely determined up to isomorphism.
\item\label{xhTalma2}If $B$ is a subalgebra of  $\alg A=\alg A_1\times\cdots\times \alg A_n$ considered above, then there exist subalgebras $B_i$ of $\alg A_i$ $(i=1,\ldots,n)$ such that $B=B_1\times\cdots\times B_n$.
\item\label{xhTalma3} Every tolerance $T$ of $\alg A$ is of the form $\,T_1\times\cdots\times T_n$ such that $T_i$ is a tolerance of $\alg A_i$ for $i=1\ldots,n$. {If $T$ is a congruence, then so are the $T_i$.}
\end{enumeratei}
\end{proposition}

Although part \eqref{xhTalma3} above is stated only for congruences in \cite{taylor}, the one-line argument ``regard $T$ as a subalgebra of $\alg A_1^2\times\cdots\times \alg A_n^2$ and apply  part \eqref{xhTalma2}'' of \cite{taylor} also works if $T$ is a tolerance rather than a congruence.

\subsection*{Results and examples} 
The properties \proponetfac--\propfourprop{} are not independent from each other and from congruence permutability. We know from H.~Werner~\cite{werner},  see also J.\,D.\,H.\ Smith~\cite{smith}, that a variety is congruence permutable if{f} it has no  proper tolerances. Obviously, a variety without proper tolerances is strongly tolerance factorable and its tolerances are the images of its congruences.   Also, we present the following statement, which generalizes the result of G.~Cz\'edli and G.~Gr\"atzer~\cite{czggg}. (The statements of this section will be proved in the next one.)

\begin{proposition}\label{czggggener}~ 
\begin{enumeratei}
\item\label{czggggeneri} Assume that $\alg A$ is a tolerance factorable algebra and $T\in\Tol{\alg A}$. Then there exist an algebra $\alg B$ $($of the same type as $\alg A)$, a congruence $\bgq$ of $\alg B$,  and a surjective homomorphism $\gf\colon \alg B\to\alg A$  such that  $T=\gf(\bgq)$, where $\gf(\bgq)=\set{(\gf(x),\gf(x)): (x,y)\in\bgq}$.  
\item\label{czggggenerii}
If a variety is strongly tolerance factorable, then its tolerances are the images of its congruences.
\end{enumeratei}
\end{proposition}

Tolerance factorability does not imply strong tolerance factorability.  For example, 
let $\var V$ be a nontrivial proper subvariety of the variety $\lat$ of all lattices. We know from G.~Cz\'edli~\cite{czg82} that $\lat$ is {strongly tolerance factorable}; see also G.~Gr\"atzer and G.\,H.~Wenzel~\cite{gggwenz} for an alternative proof. 
Consequently, $\var V$ is tolerance factorable. However, {it is not strongly tolerance factorable} by  G.~Cz\'edli~\cite[Theorem 3]{czg82}. 

Our main achievement is the following statement.

\begin{theorem}\label{thmmain} Assume that a variety $\nopvar$ is the independent join of its subvarieties $\anyvar 1,\ldots,\anyvar n$. 
{Consider one of the properties
\begin{enumeratei}
\item strong tolerance factorability,
\item tolerance factorability,
\item the tolerances of the variety are the  images of its congruences.
\end{enumeratei}
If this property holds for all the $\anyvar i$, then it also holds for $\nopvar$.}
\end{theorem}

Now we are ready to give several examples {for strongly tolerance factorable varieties with proper tolerances.} 
It would  be easy to give such examples by taking varieties equivalent to $\lat$. (For example, we could 
replace the binary join by the $n$-ary operation $f(x_1,\ldots,x_n):=x_1\vee x_2$.) 
Hence we will give pairwise non-equivalent varieties  even if Example~\ref{hnplMrl} implies  the surprising fact that strong tolerance factorability is not necessarily preserved when passing from a variety to an equivalent one.

For $2\leq n\in \mathbb N$ and $1\leq i\leq n$, let $\setni ni$ be the variety consisting of all algebras $(X,\setop n)$ such that  $X$ is a nonempty set and $\setop n$ is an $n$-ary operation symbol inducing the $i$-th projection on $X$. That is, $\setni ni$ is of type $\set{\setop n}$, and it is defined by the identity $\setop n(x_1,\ldots,x_n)=x_i$. Let $\setnn n=\setni n1\vee\cdots\vee \setni nn$ and $\setnn 1=\Set$.

\begin{example}\label{exbands} The varieties $\setnn n$, {$n\in\mathbb N$, are strongly tolerance factorable  and pairwise non-equivalent, and they have proper tolerances}.
\end{example}

Notice that $\setnn 2$ is the variety of \emph{rectangular bands}, which are idempotent semigroups satisfying the identity $xyx=x$. See A.\,H.~Clifford~\cite{clifford}, who introduced this concept, and B.~J\'onsson and C.~Tsinakis \cite{jonssontsinakis}.

Next, consider lattices with an additional unary operation $\rotop n$ {that induces an  automorphism of the lattice structure such that the identity $\rotop n^n(x)=x$ (where $\rotop n^n(x)$  denotes the $n$-fold iteration $\rotop n\bigl(\rotop n(\dots \rotop n(x)\dots)\bigr)$ of $\rotop n$) holds}. We can call them \emph{rotational lattices of order $n$}.
The variety of these lattices is denoted by $\rotlat n$. Note that $\rotlat 1$ is equivalent to $\lat$ while $\rotlat 2$ consists of \emph{lattices with involution}, which were studied, for example, in 
I.~Chajda and G.~Cz\'edli~\cite{chczinvol}. {Note also that $\rotlat n\subseteq\rotlat m$ if{f} $n\mid m$.}

\begin{example}\label{exrotlats} The varieties $\rotlat n$, $n\in\mathbb N$, {are strongly tolerance factorable  and pairwise non-equivalent, and they have proper tolerances}. Moreover, none of them is  equivalent {to a variety} given in Example~\ref{exbands}.
\end{example}

Armed with Theorem~\ref{thmmain}, one can give some more sophisticated examples. For example, we present the following. Let $\sophop$ be a binary operation symbol, and let $m,n\in\mathbb N$. We consider the type $\tau_{mn}=\set{\vee,\wedge,\rotop m,
\setop n,\sophop}$. Define the action of $\setop n$ and $\sophop$ on the algebras of $\rotlat m$ as first projections. This way these algebras become $\tau_{mn}$-algebras and they form a variety $\rottamp{\rotlat m}{n}$. 
Similarly, on the members of $\setnn n$, we define $\vee$, $\wedge$, and $\rotop m$ as first projections and  $\sophop$  as the binary second projection.
The algebras we obtain constitute a variety $\setnntamp{\setnn n}{m}$ of type $\tau_{mn}$. Let $\combvar mn={\rottamp{\rotlat m}{n}}\vee \setnntamp{\setnn n}{m}$.

\begin{example}\label{excomb}
The varieties $\combvar mn$, $m,n\in\mathbb N$, are {strongly tolerance factorable and they have proper tolerances}. Furthermore, $\combvar mn$ is equivalent to $\combvar ij$ if{f} $(i,j)=(m,n)$.
\end{example}

Note that the varieties in Example~\ref{exrotlats} are congruence distributive while those in Examples~\ref{exbands} and \ref{excomb} satisfy no nontrivial congruence lattice identity.

{Next, in the language of  lattices, we consider the ternary lattice terms 
$\tjoin(x,y,z)=x\vee(y\wedge z)$ and $\tmeet(x,y,z)=x\wedge(y\vee z)$. Clearly, the identities $x\vee y=\tjoin(x,y,y)$ and $x\wedge y=\tmeet(x,y,y)$ hold in all lattices. This 
motivates the following definition of another variety in the language of $\set{\tjoin,\tmeet}$ as follows. 
In each of the six usual laws defining $\lat$, we replace $\vee$ and $\wedge$ by
$\tjoin(x,y,y)$ and $\tmeet(x,y,y)$. For example, the absorption law $x=x\vee (x\wedge y)$ turns into the identity 
$x=\tjoin\bigl(x, \tmeet(x,y,y), \tmeet(x,y,y)\bigr)$. The six identities we obtain this way together with the identities 
$\tjoin(x,y,z)= \tjoin(x, \tmeet(y,z,z), \tmeet(y,z,z))$ and $\tmeet(x,y,z)= \tmeet(x, \tjoin(y,z,z), \tjoin(y,z,z))$ define a variety, which will be  denoted by $\tlat$.}

\begin{example}\label{hnplMrl} 
$\tlat$ is equivalent to $\lat$.  Hence
the tolerances of $\tlat$ are the images of its congruences. However, 
 $\tlat$ is not tolerance factorable. 
\end{example}

Let $\alg A\in\tlat$ and $T\in\Tol{\alg A}$. Although   {$\tlat$ is not tolerance factorable, the fact that  {it} is equivalent to {a tolerance factorable variety (which is
$\lat$) yields a natural way of defining $\alg A/T$}.  Namely, $\alg A\in \tlat$ has an alter ego $\alg A'\in \lat$ with the same tolerances, so we can take the quotient $\alg B':=\alg A'/T$ defined in $\lat$, and we can let $\alg A/T$ be the alter ego of $\alg B'$ in $\tlat$.
Clearly, the strong tolerance factorability of $\lat$ implies that  $\alg A/T\in \tlat$.}

{Since $\tlat$ is only an {``}artificial{''}  variety, we raise the following problem.}

\begin{problem}
{Is there a {well-known variety $\var V$ such that although $\var V$ is not tolerance factorable, it} 
is equivalent to {some tolerance factorable} (possibly ''artificial'') variety?}
\end{problem}

\subsection*{Proofs}
\begin{proof}[{Proof of Proposition~\ref{czggggener}}] We generalize the idea of  G.~Cz\'edli and G.~Gr\"atzer~\cite{czggg}. Assume that  $\alg A=(A,F)$ is a tolerance factorable algebra and $T\in\Tol{\alg A}$. If $\alg A$ belongs to a strongly tolerance factorable variety $\var V$, then all the algebras we construct in the proof will clearly belong to $\var V$.

The quotient algebra $\alg A/T=\bigl(\bset T,F\bigr)$, defined according to formula \eqref{mfaXCDsF}, makes sense. So does the direct product  $\alg C=\alg A\times (\alg A/T)$.
Denoting $\setm{(x,Y)\in A\times \bset T}{x\in Y}$ by $D$, the construction implies that $\alg D=(
D,F)$ is a subalgebra of $\alg C$. This $\alg D$ will play the role of $\alg B$.

Define $\bgq=\setbm{\bigl((x_1,Y_1),(x_2,Y_2)\bigr) \in {D}^2}{Y_1=Y_2  }$. As the kernel of the second projection from $\alg D$ to $\alg A/T$, it is a congruence on {$\alg D$}. 
The first projection $\gf\colon \alg D\to \alg A$, $(x,Y)\mapsto x$, is a surjective homomorphism  since, for every $x\in A$,  \eqref{HeRZrnLypp} allows us to extend $\set x$ to a block of $T$.

Clearly, if $\bigl((x_1,Y_1),(x_2,Y_2)\bigr)\in\bgq$, then 
$\set{x_1,x_2}\subseteq Y_1=Y_2\in\bset T$
implies that $\bigl(\gf(x_1,Y_1),\gf(x_2,Y_2)\bigr)= (x_1,x_2)\in T$. Conversely, assume that  $(x_1,x_2)\in T$. Then, by \eqref{HeRZrnLypp}, there is a $Y\in\bset T$ with $\set{x_1,x_2}\subseteq Y$. Hence 
$(x_1,Y), (x_2,Y)\in D$, $\bigl((x_1,Y), (x_2,Y)\bigr)\in\bgq$, and $x_i= \gf(x_i,Y)$ yield the desired equality 
$T=\setbm{\bigl(\gf(x_1,Y_1),\gf(x_2,Y_2)\bigr)} {  \bigl((x_1,Y_1),(x_2,Y_2)\bigr)\in\bgq }$.  
\end{proof}

\begin{lemma}\label{lemmanoskewblock}
Assume that $T$ is as in Proposition \textup{\ref{propIndep}}\eqref{xhTalma3} and $B\in\bset T$. 
Then there exist $B_i\in \bset{T_i}$, $i\in\set{1,\ldots,n}$, such that $B=B_1\times\cdots\times B_n$, and they are uniquely determined. Furthermore, 
$\bset T=\bset{T_1}\times\cdots\times \bset{T_n}$.
\end{lemma}

\begin{proof} Let $\pi_i$ denote the projection map $A\to A_i$, $(x_1,\ldots,x_n)\mapsto x_i$. Define $B_i:=\pi_i(B)$. First we show that $B_1\in\bset{T_1}$. If $a_1,b_1\in B_1$, then $(a_1,a_2,\ldots, a_n),(b_1,b_2,\ldots,b_n)\in B$ for some $a_j,b_j\in A_j$, $2\leq j\leq n$. Hence $B^2\subseteq T$ implies that $(a_1,b_1)\in T_1$. This gives that $B_1^2\subseteq T_1$, and we obtain $B_i^2\subseteq T_i$ for all $i\in\set{1,\ldots,n}$ by symmetric arguments. Thus 
\[(B_1\times\cdots\times B_n)^2\subseteq T_1\times\cdots\times T_n =T,
\]
which together with $B\in\bset T$ and the obvious 
$B\subseteq B_1\times\cdots\times B_n$ implies that 
\begin{equation}\label{isThx}
B = B_1 \times\cdots\times B_n\text.
\end{equation} 
The uniqueness of the $B_i$ is trivial.
If $B_1\subseteq C_1\subseteq A_1$ such that $C_1^2\subseteq T_1$, then 
\[B^2=(B_1 \times\cdots\times B_n)^2\subseteq (C_1 \times B_2\times\cdots\times B_n)^2\subseteq T_1 \times\cdots\times T_n=T\text.
\]
Hence $B\in\bset T$ yields that the first inclusion above is an equality, which implies that $B_1=C_1$. Thus $B_1\in \bset {T_1}$ and $B_i\in\bset{T_i}$ for all $i$. This together with \eqref{isThx} proves that $\bset T\subseteq \bset{T_1}\times\cdots\times \bset{T_n}$. 

Finally, to prove the converse inclusion, assume that  $U_i\in\bset{T_i}$ for $i=1,\ldots,n$, and let $U=U_1\times\cdots\times U_n$. Clearly, $U^2\subseteq T_1\times\cdots \times T_n=T$. By Zorn's lemma, there is a $B\in \bset T$ such that $U\subseteq B$.
We already know that {$B_i\in\bset{T_i}$ and \eqref{isThx} holds.}
This together with $U\subseteq B$ yields that $U_i\subseteq B_i$. Comparable blocks of $T_i$ are equal, whence $U_i=B_i$, for all $i$.  Hence $U=B\in\bset T$, proving that $\bset{T_1}\times\cdots\times \bset{T_n} \subseteq \bset T$.
\end{proof}


\begin{proof}[Proof of Theorem~\ref{thmmain}] 
{Assume first that the $\anyvar i$ are tolerance factorable.}
Let $T$ be as in  Proposition \ref{propIndep}\eqref{xhTalma3}.  
Assume that $\sterm$ is a $k$-ary term in the language of $\nopvar$ and $B_1,\ldots,B_k\in\bset T$. By Lemma~\ref{lemmanoskewblock}, there are uniquely determined $B_{ij}\in\bset{T_j}$ such that 
\begin{equation}\label{DdpOsDi}
B_i=B_{i1}\times\cdots\times B_{in}\quad\text{for}\quad i=1,\ldots,k\text.
\end{equation}
Assume that $C$ is in $\bset T$ such that 
\begin{equation}\label{oWkQm}
\setm{\sterm(b_1,\ldots,b_k)} {b_1\in B_1,\ldots, b_k\in B_k}\subseteq C\text.
\end{equation}
According to $\alg A=\alg A_1\times\cdots\times \alg A_n$, we can write $b_i=(b_{i1},\ldots,b_{in})$. 
Since $\sterm$ acts componentwise,
\begin{equation}\label{NkLSp}
\begin{aligned}
\setm{&\sterm(b_1,\ldots,b_k)} {b_1\in B_1,\ldots, b_k\in B_k}\cr
&=\setbm{\bigl(\sterm(b_{11},\ldots, b_{k1}),\ldots, \sterm(b_{1n},\ldots, b_{kn})  \bigr)    }{b_{ij}\in B_{ij}}\cr
&=\setm{\sterm(b_{11},\ldots, b_{k1}) }{b_{i1}\in B_{i1}} \times\cdots\times  \setm{\sterm(b_{1n},\ldots, b_{kn}) }{b_{in}\in B_{in}}\text.
\end{aligned}
\end{equation}
By  Lemma~\ref{lemmanoskewblock}, $C=C_1\times\cdots \times C_n$ with $C_j\in\bset{T_j}$. Combining  this with \eqref{oWkQm} and \eqref{NkLSp}, we obtain that, for $j\in\set{1,\ldots,n}$, 
\begin{equation}\label{isWeicG}
\setm{\sterm(b_{1j},\ldots, b_{kj}) }{b_{ij}\in B_{ij}\text{ for }i=1, \ldots,k}\subseteq C_j\text.
\end{equation} 
This implies the uniqueness of $C_j$ since $\anyvar j$ is tolerance factorable. Therefore, $C$ in \eqref{oWkQm} is uniquely determined, {and we obtain that $\nopvar$ is tolerance factorable.} 

{Next, assume that the $\anyvar i$ are strongly tolerance factorable. 
Observe that \eqref{isWeicG} also yields} that $C_j=\sterm(B_{1j},\ldots,B_{kj})$ in the quotient algebra $\alg A_j/T_j$. 
This, together with \eqref{DdpOsDi} and $C=C_1\times\cdots \times C_n$, implies that $\alg A/T$ is (isomorphic to)
$\alg A_1/T_1\times\cdots\times \alg A_n/T_n$. 
Since  $\anyvar j$ is {strongly} tolerance factorable, we conclude that  $\alg A_j/T_j\in\anyvar j\subseteq\nopvar$. Therefore $\alg A/T\in \nopvar$, proving that $\nopvar$ is {strongly} tolerance factorable.

{Finally, if the tolerances of $\anyvar i$ are the images of its congruences, for $i=1,\ldots, n$, then Proposition~\ref{propIndep} easily implies the same property of $\nopvar$.}
\end{proof}

\begin{proof}[Proof of Example~\ref{exbands}] 
Each of the $\setni ni$ is equivalent to $\Set$, whence it is easy to see that the $\setni ni$ are  strongly tolerance factorable. The operation $\setop n$ witnesses that $\setnn n=\setni n1\vee\cdots\vee\setni nn$ is an independent join. Hence $\setnn n$ is {strongly} tolerance factorable by Theorem~\ref{thmmain}. The three-element algebra $\alg A=\bigl(\set{a,b,c},\setop n)$, where $\setop n$ acts as the first projection, belongs to $\setni n1\subseteq \setnn n$. Consider $T\in\Tol {\alg A}$ determined by $\bset T=\bigl\{\set{ab},\set{bc} \bigr\}$. 
This $T$ witnesses that $\setnn n$ {has proper tolerances}.

Next, consider an arbitrary $\alg A\in \setnn n$. It is of the form $\alg A=\alg A_1\times\cdots\times \alg A_n$, where  $\alg A_i\in\setni ni$ for $i=1,\ldots,n$. 
Let $\sterm$ be an arbitrary term in the language of $\setnn n$. 
Since  $\setni ni$ is  equivalent to $\Set$, $\sterm$ induces a projection on $ A_i$, for $i=1,\ldots,n$. It follows that $\sterm$ induces an operation on $ A$ that  depends on at most $n$ variables.  
On the other hand, if none of the $\alg A_i$ is one-element, then $\setop n$ defines a term function  on $ A$ that depends exactly on $n$ variables.
Thus $n$ is the largest integer $k$ such that all term functions on algebras in $\setnn n$ depend on at most $k$ variables and there exists an algebra in $\setnn n$ with a term function depending exactly on $k$ variables. This proves that $\setnn n$ and $\setnn m$ are non-equivalent if $n\neq m$. 
\end{proof}

\begin{proof}[Proof of Example~\ref{exrotlats}]
Let $\alg A=(A,\vee,\wedge,\rotop n)\in
\rotlat n$ and $T\in\Tol{\alg A}$. 
Then $T$ is also a tolerance of the lattice reduct $(A,\vee,\wedge)$, and $\bset T$ for the lattice reduct is the same as it is for $\alg A$. 
{We claim that, for every  $B\in \bset T$,
\begin{equation}\label{siThKq}
\rotop n(B):=\setm{\rotop n(b)}{b\in B}\in \bset T\text.
\end{equation}
By Zorn's lemma, there is a $C\in \bset T$ such that $\setm{\rotop n(b)}{b\in B}\subseteq C$. Since $\rotop n^{-1}=\rotop n^{n-1}$ preserves $T$, $\setm{\rotop n^{-1}(c)}{c\in C}^2\subseteq T$.  This together with $B\subseteq \setm{\rotop n^{-1}(c)}{c\in C}$ and $B\in\bset T$ yields that $B= \setm{\rotop n^{-1}(c)}{c\in C}$. Therefore, $g_n(B)=C\in\bset T$, proving  \eqref{siThKq}. }

For the lattice operations, $B$ in  \eqref{mfaXCDsF} is uniquely determined 
since $\lat$ is (strongly) tolerance factorable by G. Cz\'edli~\cite{czg82}. 
By \eqref{siThKq}, the same holds for $g_n$.
Thus $\alg A/T$ makes sense. $(A/T,\vee,\wedge)$ is a lattice since  $\lat$ is {strongly} tolerance factorable.
We conclude from \eqref{siThKq} that 
$\rotop n$ is a permutation on $\alg  A/T$, whose $n$-th power is the identity map.  Finally, assume that $B\vee C=D$ in $\alg A/T$;
the case of the meet is similar.
Then, by \eqref{siThKq} and $\setm{b\vee c}{b\in B,\text{}c\in C}\subseteq D$,  
\begin{align*}
\setm{x\vee y}{x\in \rotop n(B),\text{ } y\in\rotop n(C)} 
= \setm{\rotop n(b)\vee \rotop n(c)}{b\in B,\text{ } c\in C}\cr
=\setm{\rotop n(b\vee c)}{b\in B,\text{ } c\in C} \subseteq \setm{\rotop n(d)}{d\in D} =\rotop n(D)\text.
\end{align*}
Hence $\rotop n(B)\vee\rotop n(C)=\rotop n(D)$, that is, $\rotop n$ is an automorphism of $(A/T,\vee,\wedge)$.
Therefore, $\rotlat n$ {is strongly tolerance factorable. It has proper tolerances since so has $\lat$, which is equivalent to the subvariety $\rotlat 1$ of $\rotlat  n$. } 

The boolean lattice with $n$ atoms allows an automorphism  $\gf$ of order $n$ such that the subgroup generated by $\gf$ acts transitively on the set of atoms, but no such automorphism of smaller  order is possible. This implies easily that $\rotlat m$ is not equivalent to $\rotlat k$ if $m\neq k$. Since $\rotlat n$ is congruence distributive, it is not equivalent to $\setnn m$.
\end{proof}

\begin{proof}[Proof of Example~\ref{excomb}]
Since $\sophop$ takes care of independence, 
Examples~\ref{exbands} and \ref{exrotlats} together with Theorem~\ref{thmmain} yield that  $\combvar mn$ {is strongly tolerance factorable and it has proper tolerances}. Suppose for a contradiction that $(m,n)\neq (u,v)$ but  $\combvar mn$ is equivalent to  $\combvar uv$.

Suppose first that $m=u$ and $n\neq v$. Let, say, $v<n$. 
Take the $2^n$-element  
$\alg A\in  \setnntamp{\setnn n}{m}  \subseteq \combvar mn$ for which all the $\alg A_i$ in 
Proposition \ref{propIndep}\eqref{xhTalma1} are 2-element. Let $\sterm$ be a binary term in the language of $\combvar mn$. Since all terms induce projections on $\alg A_i$, the identity 
$\sterm(x,\sterm(y,x))=x$ 
holds in $\alg A_i$ for $i=1,\ldots,n$. Therefore, $\alg A$ satisfies the same identity, for every binary term $\sterm$.
Observe that, up to now, we did not use the assumption on the size of $\alg A_i$,  whence  
\begin{equation}\label{sTmidbInhowk}\sterm(x,\sterm(y,x))=x \text{ holds in }\setnn n\text{, for all binary terms }s\text.
\end{equation}%
By the assumption, there is a 
$\combvar mv${-}structure $\alg B$ on the set $A$ such that $\alg B$ and $\alg A$ have the same term functions. 
By the definition of $\combvar mv{= \combvar uv}$, 
$\alg B$ is (isomorphic to) $\alg C\times \alg D$, where $\alg C\in {{\rottamp{\rotlat m}{v}}}$  and $\alg D\in { { \setnntamp{\setnn v}{m} } }$. 
Since $\alg C$ is a homomorphic image of $\alg B$ and $\alg B$ has the same term functions as $\alg A$, the identity $\sterm(x,\sterm(y,x))=x$ holds in $\alg C$ for all binary terms $\sterm$.  Thus $\alg C$ is one-element since otherwise $\sterm(x,y)=x\vee y$ would fail this identity. Hence the term functions of $\alg B$ are the same as those of its $\setnn v$-reduct. 
Now, we can obtain a contradiction the same way as in the last paragraph of the proof of Example~\ref{exbands}: $\alg A$ has an $n$-ary term function that depends on all of its variables while all term functions of $\alg B$ depend on at most $v$ variables. This proves that $n=v$.

Secondly, we suppose  that $m\neq u$. Let, say, $m>u$. Consider the algebra $\alg A\in {\rottamp{\rotlat m}{n}}\subseteq \combvar mn$ such that 
the $\rotlat m$-reduct of $\alg A$ is the $2^m$-element boolean lattice and $g_m$ is a lattice automorphism of order $m$ that acts transitively on the set of atoms. (That is, the restriction of $g_m$ to the set of atoms is a cyclic permutation of order $m$.)
Since $\combvar mn$ is equivalent to $\combvar un=\combvar uv$, there exist algebras $\alg C\in {\rottamp{\rotlat u}n}$ and $\alg D\in {\setnntamp{\setnn n}u}$ such that 
$\alg B:=\alg C\times \alg D\in \combvar un$ is equivalent to $\alg A$.  Observe that $\alg D$, which is a homomorphic image of $\alg B$, has a lattice reduct.  Hence, like in the firts part of the proof, \eqref{sTmidbInhowk} easily implies that $\alg D$ is a one-element algebra. Therefore,  $\alg A$ is equivalent to $\alg C$, that is, to a member of  ${\rottamp{\rotlat u}{n}}$. Hence the $\rotlat m$-reduct of $\alg A$ is equivalent to a member of $\rotlat u$. This leads to a contradiction the same way as in the last paragraph of the proof of Example~\ref{exrotlats}.
\end{proof}

\begin{proof}[{Proof of Example~\ref{hnplMrl}}] {Consider the lattice $L$ in Figure~\ref{figone} as an algebra of $\tlat$. A tolerance $T\in\Tol L$ is given by its blocks $A=[a_0,a_1]$, \dots, $E=[e_0,e_1]$. (It is easy to check, and it follows even more easily  from G.~Cz\'edli~\cite[Theorem 2]{czg82}, that $T$ is a tolerance.)}
Since 
\[\setm{\tjoin(x,y,z)}{x\in A,\,\, y\in B,\,\, z\in C}=[c_0,a_1],
\]
{this set} is a subset of two {distinct} blocks, $A$ and $C$. Hence $\tlat$ is not tolerance factorable. The rest is trivial. 
\end{proof}

\begin{figure}
\includegraphics[scale=1.0]{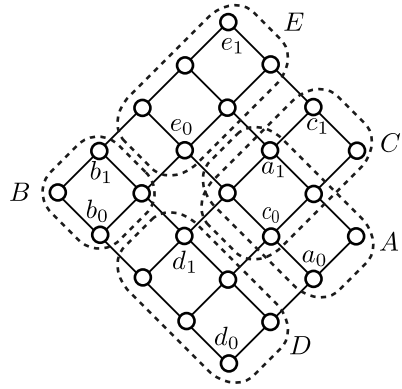}
\caption{$L$ and the blocks of $T$} 
\label{figone}
\end{figure}

\begin{ackno} The authors thank 
Paolo Lipparini for helpful comments and for calling their attention to H.\ Werner~\cite{werner}.
\end{ackno}

\end{document}